\documentclass[10pt,leqno]{amsart}

\usepackage[backref=page]{hyperref}

\renewcommand*{\backref}[1]{}
\renewcommand*{\backrefalt}[4]{%
	\ifcase #1 Not cited.%
	\or        Cited on page~#2.%
	\else      Cited on pages~#2.%
	\fi}

\usepackage[a4paper,lmargin=2.0cm,rmargin=2.0cm,tmargin=2.0cm,bmargin=2.0cm]{geometry}

\usepackage{parskip}

\usepackage[utf8]{inputenc}
\usepackage{amssymb}
\usepackage{amsmath}
\usepackage{amsthm}
\usepackage{rotating}
\usepackage{color}
\usepackage[english]{babel}

\usepackage{amsfonts}
\usepackage{indentfirst}
\usepackage{framed}
\usepackage{float}

\usepackage{listings}

\usepackage{array}
\usepackage{url}

\newtheorem{thm}{Theorem}[section]
\newtheorem*{thm*}{Theorem}
\newtheorem{lem}[thm]{Lemma}

\newtheorem{conj}[thm]{Conjecture}

\theoremstyle{definition}

\theoremstyle{remark}
\newtheorem{rmk}[thm]{Remark}

\title{Primitive weird numbers having more than three distinct prime factors}

\author[G. Amato]{Gianluca Amato}
\address{Universit\`a di Chieti-Pescara\\ Dipartimento di Economia Aziendale, viale della Pineta 4, I-65129 Pescara, Italy}
\email{gianluca.amato@unich.it}
\author[M. F. Hasler]{Maximilian F. Hasler}
\address{Universit\'e des Antilles\\D.S.I., B.P. 7209\\ Campus de Schoelcher\\ F-97275 Schoelcher cedex, Martinique (F.W.I.)}
\email{Maximilian.Hasler@Univ-Antilles.fr}
\author[G. Melfi]{Giuseppe Melfi}
\address{University of Applied Sciences of Western Switzerland\\ HEG Arc\\
	Espace de l'Europe, 21, CH-2000 Neuch\^atel}
\email{giuseppe.melfi@he-arc.ch}
\author[M. Parton]{Maurizio Parton}
\address{Universit\`a di Chieti-Pescara\\ Dipartimento di Economia, viale della Pineta 4, I-65129 Pescara, Italy}
\email{maurizio.parton@unich.it}

\thanks{Part of this research was done while the first author was visiting the Department of Mathematics and Computer Science at Wesleyan University.}



\newenvironment{pfma}{\indent\rm Proof of \ref{thm:main}.1}{\hfill $\square$ \par \indent}
\newenvironment{pfmb}{\indent\rm Proof of \ref{thm:main}.2}{\hfill $\square$ \par \indent}

\newcommand{\NN}{\mathbb{N}}
\newcommand{\ug}{=}

\begin{document}

\begin{abstract}
In this paper we study some structure properties of pri\-\-mi\-tive weird numbers
in terms of their factorization. We give sufficient conditions to 
ensure that a positive integer is weird. Two algorithms for generating 
weird numbers having a given number of distinct prime factors 
are presented. These algorithms yield primitive weird numbers of the form $mp_1\dots p_k$ for a suitable deficient positive integer $m$ and primes $p_1,\dots,p_k$ and generalize a recent technique developed for generating primitive weird numbers of the form $2^np_1p_2$.
The same techniques can be used to search for odd weird numbers, whose existence is still an open question.
\end{abstract}

\maketitle

\section{Introduction}

Let $n\in\NN$ be a natural number, and let $\sigma(n)\ug\sum_{d|n}d$ be the sum of its divisors. 
If $\sigma(n)>2n$, then $n$ is called \emph{abundant}, whereas if $\sigma(n)<2n$, then $n$ is called \emph{deficient}. \emph{Perfect numbers} are those for which $\sigma(n)=2n$. According to \cite{Guy}, we will refer to $\Delta(n)\ug\sigma(n)-2n$ as the \emph{abundance} of $n$, and to $d(n)\ug 2n-\sigma(n)=-\Delta(n)$ as the \emph{deficience} of $n$.
If $n$ can be expressed as a sum of distinct proper divisors, then $n$ is called \emph{semiperfect}, or sometimes also \emph{pseudoperfect}. Slightly abundant numbers with $\Delta(n)=1$ are called \emph{quasi-perfect}, and slightly deficient numbers with $d(n)=1$ are called \emph{almost perfect}.

A \emph{weird number} is a number which is abundant but not semiperfect. In other words, $n\in\NN$ is weird if it is abundant and it cannot be written as the sum of some 
of its proper divisors.

Weird numbers have been defined in 1972 by Benkoski \cite{Ben}, and appear to be rare: for instance, up to $10^4$ we have only $7$ of them \cite{Slo}.
Despite this apparent rarity, which is the reason of the name, weird numbers are easily proven to be infinite: if $n$ is weird and $p$ is a prime larger than $\sigma(n)$, then $np$ is weird (see for example \cite[page~332]{Fri}).
But a much stronger property is true: Benkoski and Erd\H os, in their joint 1974 paper \cite{Ben-Erd}, proved that the set of weird numbers has positive asymptotic density.

Several questions on weird numbers have not been settled yet. For instance, if we look for \emph{primitive weird numbers}, that is, that are not multiple of other weird numbers, we don't know whether they are infinite or not:
\begin{conj}\cite[end of page 621]{Ben-Erd}
There exist infinitely many primitive weird numbers.
\end{conj}
In this respect, the third author recently proved in \cite{Mel} that the infiniteness of primitive weird numbers follows by assuming the classic Cramér conjecture on gaps between consecutive primes~\cite{Cra}.

Another open question is the existence of odd weird numbers.
Erd\H os offered \$\,10 for an example of odd weird number, and \$\,25 for a proof that none can exist~\cite{Ben}. Recently Wenjie Fang~\cite[Sequence A006037]{Slo} claimed that there are no odd weird numbers up to $10^{21}$, and no odd weird numbers up to $10^{28}$ with abundance not exceeding $10^{14}$.

Moreover, very little is known about a pattern in the prime factorization of primitive weird numbers. As of now, most of the known primitive weird numbers are of the form $2^npq$ with $p$ and $q$ primes, and all the papers on primitive weird numbers deal with numbers of this form \cite{Ian,Kra,Mel,Paj}. Relatively few examples of primitive weird numbers with more than three distinct prime factors are known up to now: for instance among the 657 primitive weird numbers not exceeding $1.8\cdot10^{11}$ there are 531 primitive weird numbers having three distinct prime factors; 69  having four distinct prime factors; 54 having five distinct prime factors and only 3 having six distinct prime factors. 

This paper considers primitive weird numbers that have several distinct prime factors.
In particular, we give sufficient conditions in order to ensure that a positive integer of the form $mp_1\dots p_k$ is weird, where $m$ is a deficient number and $p_1,\dots,p_k$ are primes (see Theorem~\ref{thm:main}).

We then apply Theorem~\ref{thm:main} to search for new primitive weird numbers, looking in particular at four or more prime factors. We find hundreds of primitive weird numbers with four distinct prime factors of the form $2^mp_1p_2p_3$, $75$ primitive weird numbers with five distinct prime factors of the form $2^mp_1p_2p_3p_4$, and $9$ primitive weird numbers with six distinct prime factors (see Section~\ref{sec:app}).

This paper generalizes to several factors a technique developed in \cite{Mel}. This approach, as far as we know, is the first that can be used to generate primitive weird numbers with several distinct prime factors. Moreover, since there are many odd deficient numbers, Theorem~\ref{thm:main} can be used to hunt for 
the first example of an odd weird number (see Section~\ref{sec:odd}).

\section{Basic ideas}

We recall a fundamental lemma that will be extensively used, and that corresponds to an equivalent definition of weird number.

\begin{lem}\label{fundam}
An abundant number $w$ is weird if and only if $\Delta(w)$ cannot 
be expressed as a sum of distinct divisors of $w$.
\end{lem}
\begin{proof}
For a proof one can see \cite[Lemma 2]{Mel}.
\end{proof}

We will need another technical lemma, which will be used in the proof of the main theorems.

%
%
%
%

\begin{lem}
\label{primitiveweird}
If $w = mq$ is an abundant number, $m$ is deficient, $q$ is prime and $q  \geq \sigma(p^\alpha) - 1$ for each $p^\alpha || m$, then $w$ is primitive abundant.
\end{lem}
\begin{proof}
Since a multiple of an abundant number is abundant, in order to prove that $m$ is a primitive abundant number, (i.e., an abundant number whose proper divisors are all deficient), it suffices to prove that $w/p$ is deficient for each $p|w$. If $p = q$, then $w/q = m$, and we assumed that $m$ is deficient. Otherwise, if $p^\alpha||m$, then
\begin{eqnarray*}
\frac{\sigma(w/p)}{w/p} & = &
\displaystyle
\frac{\sigma(w)(p^\alpha-1)p}{w(p^{\alpha+1}-1)} = \frac{\sigma(w)}{w}\cdot\left(1-\frac{p-1}{p^{\alpha+1}-1}\right)\\
&\le&\displaystyle
 \frac{\sigma(w)}{w}\cdot \left(1-\frac{1}{q+1}\right) = \frac{\sigma(w)}{w} \cdot \frac{q}{q+1} = \frac{\sigma(w/q)}{w/q}  < 2.
\end{eqnarray*}\vskip-3ex
\end{proof}

\section{Main result}\label{sec:main}

In this section we 
provide two ways for generating primitive weird numbers.  


%
%
%
%
%
%
%
%
%

\begin{thm}\label{thm:main}
Let $m>1$ be a deficient number
and $k>1$.
Let $p_1,\dots,p_k$ be primes 
with $\sigma(m)+1<p_1<\dots< p_k$. Let 
$$
 h^*= \left[\frac{p_1-\sigma(m)}{p_k-p_1}\right].
$$
Let $\tilde{w}=mp_1\dots p_k$, and
\begin{equation}
\label{union}
U_{m,p_1,p_k}:=\bigcup_{j=0}^{ h^* }
\{n\in\mathbb N ~\mid~ jp_k+\sigma(m)<n<(j+1)p_1\}.
\end{equation}
\begin{itemize}
\item[\ref{thm:main}.1]
If $\tilde{w}$ is abundant and $\Delta(\tilde{w})\in U_{m,p_1,p_k}$, 
then $w=\tilde{w}$ is a primitive weird number.
\item[\ref{thm:main}.2]
If $\tilde{w}$ is deficient, let
$$p<\frac{2\tilde{w}}{d(\tilde{w})}-1$$
a prime with $p>p_k$. 
Then $w=\tilde{w}p$ is abundant. Furthermore, if $\Delta(w)\in U_{m,p_1,p_k}$, 
then
$$p>\frac{2\tilde{w}}{d(\tilde{w})}-1-
\frac{\displaystyle(1+h^*)p_1}{d(\tilde{w})}$$
and moreover, if $p>\Delta(w)$ then
$w$ is a primitive weird number.
\end{itemize}
\end{thm}

\begin{pfma}
%
Since $\sigma(m)+2\le p_1$, the set union in the right side 
of~(\ref{union}) is not empty. The sets of consecutive integers 
involved in the union in the right side of~(\ref{union}) 
are pairwise disjoint.
If
 $$ h <\frac{p_1-\sigma(m)}{p_k-p_1},  $$
then
$$\sigma(m)+hp_k<(h+1)p_1.$$

Now, let $j\le h^*$ and let $n$ be an integer with $jp_{k}+\sigma(m)<n<(j+1)p_1$.
We will prove that $n$ cannot be expressible as a sum of distinct 
divisors of $w$.

Note that $n<p_1^2$. This is because $n < (j+1)p_1  \le (h^*+1)p_1 <
(p_1/2 + 1)p_1$ and $p_1/2 +1 < p_1$.
This means that if $n$ is expressible as a sum of 
distinct divisors of $w$, these divisors must be of the form 
$dp$ with $d\mid m$ and $p\in\{p_1,\dots,p_{k}\}$, 
or simply of the form $d$, with $d\mid m$. Let's say 
$n=d_1p_1+\dots+d_Np_N+d_1'+d_2'+\dots+d_M'$, where $d_1',\dots,d_M'$ 
are distinct divisors of $m$.
Then, necessarily $d_1+d_2+\dots +d_N\le j$, since $n<(j+1)p_1$. 
As a consequence, we have:
\[
d_1'+\dots+d_M'=n-(d_1p_1+\dots d_Np_N)\ge n-jp_k>\sigma(m)
\]
and this is in contradiction with the assumption on $d_1', \dots, d_M'$.
So $n$ cannot be expressible as a sum of distinct divisors of $w$. 

No elements in $U_{m,p_1,p_k}$ can be expressed as a sum of distinct divisors of $w$.
Since $\Delta(w)\in U_{m,p_1,p_k}$, by Lemma~\ref{fundam} this implies that 
$w$ is weird.

In order to prove that $w$ is a primitive weird number, by 
Lemma~\ref{primitiveweird} it suffices to prove that $w/p_k$ is deficient.
If $k=2$, then $\Delta(w/p_k)=\Delta(mp_1)=\Delta(m)p_1+\sigma(m)$. Since 
$\Delta(m)\leq-1$ and $p_1>\sigma(m)$, then $\Delta(w/p_2)<0$. 
We may assume that $k\ge3.$ Then we have:%
%
\begin{eqnarray*}
\Delta\left(\frac w{p_k}\right) & = &
\sigma\left(\frac w{p_k}\right)-2\frac w{p_k} ~=~
\frac{\sigma(w)}{p_k+1}-\frac{2w}{p_k}
\\&=& \frac{p_k(\sigma(w)-2w)-2w}{p_k(p_k+1)}
~=~\frac{\Delta(w)-\displaystyle\frac{2w}{p_k}}{p_k+1}
\end{eqnarray*}
Now, $\Delta(w)\in U_{m,p_1,p_k}$. Since $k\ge3$, between $p_1$ and $p_k$ there is at least an odd integer that is not prime, and therefore $h^*<p_1/2k$. This means that
$$\Delta(w)<(h^*+1)p_1< \left(1+\frac{p_1}{2k}\right)p_1.$$ 
On the other hand $2w/p_k\ge2mp_1^2$, and since
$2mp_1> p_1+p_1>1+p_1/2k$, this means that $\Delta(w)<2w/p_k$ and therefore 
$\Delta(w/p_k)<0$.
In particular, since $w/{p_k}$ is deficient, by Lemma~\ref{primitiveweird}, 
$w$ is a primitive weird number.
\end{pfma}

\begin{pfmb}
Note that $p$ is the largest prime that divides $w$, and that 
$w/p=\tilde{w}$ is deficient. So by Lemma~\ref{primitiveweird}, 
in order to prove that 
$w$ is a primitive weird number,  
it suffices to prove that $w$ is indeed abundant and weird. 

Since $(\tilde{w},p)=1$ and $2\tilde{w}-d(\tilde{w})p-d(\tilde{w})> 0$ by hypothesis, we have
\begin{eqnarray*}
\Delta(w) & = & \sigma(w)-2w\\
&=&\sigma(\tilde{w})(p+1)-2p\tilde{w}\\
&=&(\sigma(\tilde{w})-2\tilde{w})p+\sigma(\tilde{w})\\
&=&-d(\tilde{w})p+2\tilde{w}-d(\tilde{w}) ~>~ 0
\end{eqnarray*}
This proves that $w$ is abundant. 

Now assume that $\Delta(w)\in U_{m,p_1,p_k}$. 
Since $\max U_{m,p_1,p_k}<(1+h^*)p_1$, then 
$\Delta(w) = -d(\tilde{w})p+2\tilde{w}-d(\tilde{w})<(1+h^*)p_1 $
and therefore
$$p>\frac{2\tilde{w}}{d(\tilde{w})}-1-
\frac{(1+h^*)p_1}{d(\tilde{w})}$$

Let $p>\Delta(w)$. 
In order to prove that $w$ is weird, by Lemma~\ref{fundam}, we have to prove 
that $\Delta(w)$ 
is not a sum of proper divisors of $w$.

Since $\Delta(w)<p$, if $\Delta(w)$ is a sum of proper divisors of $w$, 
then all divisors involved must be divisors of $\tilde{w}$. 
On the other hand $\Delta(w)\in U_{m,p_1,p_k}$ and as seen above, 
no element in $U_{m,p_1,p_k}$ can be expressed as a sum of distinct 
divisors of $\tilde{w}$. This completes the proof.
\end{pfmb}

\begin{rmk}
Very often, when conditions of Theorem~\ref{thm:main}.2 hold, it is
$(1+h^*)p_1<d(\tilde{w}).$
This means that in these cases, $p=[ 2\tilde{w}/d(\tilde{w})-1 ]$.
\end{rmk}

\begin{rmk}
If $m, p_1, \ldots, p_k$ is a sequence such that $w = mp_1 \dots p_k$ 
is primitive weird according to Theorem~\ref{thm:main}.1, $k > 2$ and 
$\Delta(w) < p_k$, then $\tilde{w}=p_1\dots p_{k-1}$ and $p=p_k$ verify the conditions of Theorem~\ref{thm:main}.2
\end{rmk}

Indeed, if $w$ satisfies the condition of Theorem~\ref{thm:main}.1, then $w$ is abundant and $\tilde{w} = w/p_k$ is
deficient. This implies that $p_k < 2\tilde{w}/d(\tilde{w}) - 1$. 
Moreover, since
$\Delta(w) \in U_{m,p_1,p_k}$, there is a $j \le h^* = \left[\frac{p_1 - \sigma(m)}{p_k - p_1}\right]$ such
that $jp_k +
\sigma(m) < \Delta(w) < (j+1)p_1$. Since $p_k > p_{k-1}$, then $jp_{k-1} + \sigma(m) < \Delta(w)$, with $j \le \Big[ \frac{p_1 - \sigma(m)}{p_{k-1} - p_1}\Big]$. This means $\Delta(w) \in U_{m,p_1,p_{k-1}}$ as for
Theorem~\ref{thm:main}.2. Finally, if $\Delta(w) < p_k$ all requirements of Theorem~\ref{thm:main}.2 are satisfied.

Despite of the above remark, the conditions of Theorem~\ref{thm:main}.1. and~\ref{thm:main}.2 are not equivalent. For example, 
$m=2^{11}$, $p_1=11321$, $p_2=12583$ and $p_3=13093$  verify 
conditions \ref{thm:main}.1 (so $w=mp_1p_2p_3$ is a primitive weird number)
but not \ref{thm:main}.2, because $p_3<\Delta(w)=43936$. In the other sense, for $m=2^{12}$,
$p_1=23143$, $p_2=24043$, $p_3=27061$ and $p=3077507$, 
conditions \ref{thm:main}.2 are satisfied, but the weird number 
$\tilde{w}=mp_1p_2p_3p_4$ does not verify the conditions \ref{thm:main}.1, 
because $\Delta(\tilde{w})=39680\not\in U_{m,p_1,p_4}$.

\section{The application of Theorem~\ref{thm:main}}\label{sec:app}

Theorem~\ref{thm:main} may be used to develop an algorithm which searches for primitive weird numbers with many different prime factors. The sufficient conditions are computationally much easier to check than the standard definition of weird number. The question, however, is in which range the prime numbers $p_1, \ldots, p_k$ should be chosen. The following theorem gives a partial answer.

\begin{thm}\label{abudef}
Let $m$ be a deficient number and let $d$ be its deficience. Let 
$p_1,\dots,p_k$ be primes, with $m<p_1<p_2<\dots<p_k$.
Let $\tilde{w}=mp_1\dots p_k$.
\begin{itemize}
\item[(i)] If $p_k<2km/d-(k+2)/2$ then $\tilde{w}$ is abundant.
\item[(ii)] If $p_1>2km/d-k/2$ then $\tilde{w}$ is deficient.
\end{itemize}
\end{thm} 

\begin{proof}
\textsl{(i).}
We first prove that if $p_1,p_2,\dots ,p_k$ are distinct primes with 
$m<p_1<\dots<p_k< 2km/d-(k+2)/2$ then
$w=mp_1\dots p_k$ is abundant.

Since $(m,p_i)=1$, one has
$\sigma(w)=\sigma(m)\sigma(p_1)\dots\sigma(p_k)$.
This implies that:
$$
\frac{\sigma(w)}w=
\left(2-\frac dm\right)\prod_{h=1}^k\left(1+\frac1{p_h}\right)
>\left(2-\frac dm\right)\cdot\left(1+\frac1{p_k}\right)^k>2.
$$

The above inequalities hold because for positive $x$, 
the function $x\rightarrow1+1/x$ is decreasing, and because the equation
$$
\left(2-\frac dm\right)\cdot\left(1+\frac1q\right)^k=2
$$
holds for 
$$q=\frac1{\displaystyle\sqrt[k]{\frac{2}{2-\frac dm}}-1}
=\frac{2k}dm-\frac{k+1}2+O\left(\frac dm\right)>p_k.$$

Therefore $w$ is abundant.

\textsl{(ii).} The proof is analogous.
\end{proof}

If one wants to generate weird numbers by means 
of Theorem~\ref{thm:main} (\ref{thm:main}.1 or \ref{thm:main}.2), i.e.,
abundant numbers $\tilde{w}$ with  
$\Delta(\tilde{w})\in U_{m,p_1,p_k}$ 
on a hand $U_{m,p_1,p_k}$ 
have to be as large as possible. Good choices are with 
$p_1-\sigma(m)$ as large as possible. On the other hand, in order to get higher values of $h^*$, $p_k-p_1$ have to be as small as possible. 
This leads to consider $k$-tuples of primes
$p_1,\dots p_k$ in an interval that, by Theorem~\ref{abudef}, includes $2km/d-(k+1)/2$. 
However, if $k \leq d$ then $2km/d$ might be smaller than $\sigma(m)$, and if $p_1 < \sigma(m)$ then $h^* < 0$ and $U_{m,p_1,p_k}$ is empty. Therefore, $k > d$ is generally preferable, and all new weird numbers
we have found enjoy this property.


The primitive weird numbers generated with Theorem 1 in \cite{Mel} 
are particular cases of Theorem~\ref{thm:main}.1, with $m=2^h$, 
and $k=2$. When $p_1$ and $p_2$ are chosen according to that theorem, then conditions of Theorem~\ref{thm:main}.1 are fulfilled and $w=2^hp_1p_2$ 
is a primitive weird number.

However,
Theorems~\ref{thm:main}.1 and~\ref{thm:main}.2 become more interesting 
when applied to search weird numbers with several prime factors.


 
Theorem \ref{thm:main}.1 yields primitive weird numbers of the form 
$mp_1\dots p_k$ where $m$ is a deficient number, 
$k$ is an integer larger than the deficience of $m$ and the $p_i$'s are
suitably chosen. It is relatively easy to generate primitive 
weird numbers up to four distinct prime factors. 
The table below shows some of the primitive weird numbers having at least five 
distinct prime factors we were able to generate with Theorem \ref{thm:main}.1

\begin{table}[H]
\begin{center}
\scriptsize
\begin{tabular}{|r|l|l|r|}
\hline
$w$ & prime factorization & $m$ & $\Delta(w)$ \\
\hline 
9210347984 & $2^4\cdot83\cdot89\cdot149\cdot523$ & $2^4$ &  32\\
9772585048 & $2^3\cdot17\cdot 317\cdot419\cdot541 $ & $2^3 \cdot 17$ & 304 \\
23941578736 & $2^4\cdot73\cdot103\cdot127\cdot1567$ & $2^4$ & 32\\
109170719992 & $2^3\cdot19\cdot79\cdot2731\cdot3329$ & $2^3$ & 16\\
359214428128 & $2^5\cdot127\cdot211\cdot509\cdot823$ & $2^5$ & 64\\
446615164768 & $2^5\cdot163\cdot167\cdot331\cdot1549$ & $2^5$ & 64\\
83701780710848 & $2^6\cdot181\cdot563\cdot3407\cdot3767$ & $2^6$ & 128 \\
823548808494656 & $2^6\cdot139\cdot3631\cdot4441\cdot5741$ & $2^6$ & 128 \\
31871420410521385088 & $2^7\cdot257\cdot90803\cdot98221\cdot108631$ & $2^7 \cdot 257$ & 78464\\
32852586770937891968 & $2^7\cdot257\cdot87443\cdot90803\cdot125777$ & $2^7 \cdot 257$ & 85184 \\
32892333375893455232 & $2^7\cdot257\cdot79841\cdot109943\cdot113909$ & $2^7 \cdot 257$ & 76736 \\
33622208489084493184 & $2^7\cdot257\cdot76757\cdot107873\cdot123439$ & $2^7 \cdot 257$ & 72832\\
\hline
\end{tabular}
\end{center}
\caption{\label{tab:teo1}The above primitive weird numbers have five distinct prime 
factors and have been found by means of Theorem~\ref{thm:main}.1 with 
$k=4$ for $m = 2^h$ and with $k=3$ for $m=136 = 2^3 \cdot 17$ or $m=32896= 2^7 \cdot 257$.}
\end{table}



An implementation of Theorem~\ref{thm:main}.2 yields in minutes hundreds 
of primitive weird numbers of the form $2^hp_1p_2p$. 
By going deeper in the implementation of Theorem~\ref{thm:main}.2, we 
have been able to find 65 primitive weird numbers $w$ 
having five distinct prime factors of the form 
$w=2^hp_1p_2p_3p$; and nine primitive weird numbers having 
six distinct prime factors, that are shown in the table below. 
As a comparison, before our computations only three primitive weird numbers 
having six distinct prime factors were known at the OEIS database \cite{Slo}.

\begin{table}[H]
\begin{center}
\scriptsize
\begin{tabular}{|r|l|r|}
\hline
$w$ & prime factorization & $\Delta(w)$ \\
\hline 
44257207676 & $2^2\cdot11\cdot37\cdot59\cdot523\cdot881 $& 8 \\
125258675788784 & $2^4\cdot47\cdot149\cdot353\cdot1307\cdot2423$ & 32 \\
147578947676144 & $2^4\cdot43\cdot211\cdot367\cdot1091\cdot2539$ & 32 \\
4289395775422432 & $2^5\cdot127\cdot211\cdot401\cdot1861\cdot6703$ & 64 \\
5976833582079328 & $2^5\cdot181\cdot197\cdot353\cdot431\cdot34429$ & 64 \\
1663944565537013728 & $2^5\cdot131\cdot223\cdot311\cdot2179\cdot2626607$ & 64 \\
206177959637947617894769024 & $2^7\cdot257\cdot84691\cdot101891\cdot116041\cdot6259139$ & 74752 \\
2996153601600440129026407808 & $2^7\cdot257\cdot89021\cdot93239\cdot118621\cdot92505877$ & 83584 \\
48083019473926272314825065088 & $2^7\cdot257\cdot97213\cdot97973\cdot100957\cdot1520132521$ &  287264\\
\hline
\end{tabular}
\end{center}
\caption{\label{tab:teo2}
The above primitive weird numbers have six distinct prime 
factors and have been found by means of Theorem~\ref{thm:main}.2 with $k=4$ for
$m=2^h$ and $k=3$ for $m=32896=2^7 \cdot 257$}
\end{table}


\section{Tracking weird numbers with several distinct prime factors and eventual odd weird numbers}\label{sec:odd}

As we have seen, Theorems~\ref{thm:main}.1 and \ref{thm:main}.2 
provide two distinct 
strategies to track primitive weird numbers 
with several distinct prime factors. The same approaches could be applied 
to track odd weird numbers. If $m>1$ is odd, the integers $w$
that one creates by means of Theorem~\ref{thm:main}.1 or \ref{thm:main}.2 are odd, 
and if the conditions are fulfilled, $w$ would be a primitive odd weird number.

For tracking a weird number with several distinct prime factors (even or odd)
both strategies (\ref{thm:main}.1 and \ref{thm:main}.2) start with choosing a deficient number 
$m$ with deficience $d$.
As a general rule, since we want $k>d$,
small values of $d$ have to be preferred in order to keep the computational complexity low. Indeed, 
the case 
$d=1$ corresponds to $m=2^h$ assuming that no further 
almost perfect numbers exist. 
This case has been largely discussed in \cite{Ian,Mel}. 
The only known composite numbers with $d=2$ are even. 
For example, $m=136$, $m=32896$  have deficience 
2~\cite[Sequence A191363]{Slo} and, as shown in Table~\ref{tab:teo1} and Table~\ref{tab:teo2}, 
Theorem~\ref{thm:main} allows to find several primitive weird 
numbers starting with such values of $m$. We expect that 
primitive weird numbers could be generated also from $m=2147516416$,
whose deficience is 2, both with Theorem \ref{thm:main}.1 and \ref{thm:main}.2. Unfortunately the computation becomes dramatically longer.

As far as we know there are no known integers with deficience $d=3$. The only known integers with deficience $d=4$ are even, and the only known integer with $d=5$ is 9 for which the above approaches are hard to apply.  

%

An interesting case is $d=6$. There are several integers whose deficience is 6,
some of which are odd. The list starts with $7$, $15$, $52$, $315$, $592$, 
$1155$, $2102272$, $815634435$, 
and no other terms are known \cite[Sequence A141548]{Slo}. 

So, if one wants to track odd weird numbers, the starting $m$ in 
both approaches~\ref{thm:main}.1 and~\ref{thm:main}.2, could be $m=315$, $m=1155$ or $m=815634435$. Unfortunately our attempts to track an odd weird numbers from such a choice of $m$ have been unfruitful.

\section{Conclusion}

As one can argue from a table of primitive weird numbers, 
most of the primitive weird numbers are of the form $2^kpq$ for 
$k\in\mathbb N$, and primes $p$ and $q$. This was already pointed out in 
\cite{Mel} where the third
author, among other things, conjectured that there 
are infinitely many primitive weird numbers of the form $2^kpq$.  

It seems that primitive weird numbers 
that are not of this form become rarer.
For example, between 
the 301st and the 400th, there are only 7 primitive weird numbers 
that are not of the form $2^kpq$. Five of them have four distinct prime factors and two of them have five distinct prime factors. The existence of several weird numbers having six distinct prime factors leads to the following conjecture.

\begin{conj}
\label{omega3}
Given an integer $k\ge3$, there exists a primitive weird number having at least $k$ distinct prime factors.
\end{conj} 


Of course, a positive answer would settle the question 
of the infiniteness of primitive weird numbers. 



However, a 
proof that
primitive weird numbers have a bounded number of distinct prime factors 
would not settle neither the question of the infiniteness of primitive 
weird numbers nor the question of the existence of odd weird numbers.

\end{document}